\documentclass[12pt, a4paper, reqno]{amsart}
\usepackage[english]{babel}
\usepackage{amsmath,amsfonts,amssymb,amsthm}
\usepackage{epsfig}
\usepackage{pgf}
\usepackage{hyperref}
\usepackage{graphicx,graphics}
\usepackage{epstopdf}
\usepackage{color}
\usepackage{caption}
\usepackage{subcaption}
\usepackage{longtable}

\usepackage[margin=2.5cm]{geometry}

\newtheorem{thm}{Theorem}[section]
\newtheorem{lem}[thm]{Lemma}
\newtheorem{cor}[thm]{Corollary}

\theoremstyle{definition}
\newtheorem{defin}[thm]{Definition}
\newtheorem{exam}[thm]{Example}
\theoremstyle{remark}

\newcommand { \ib }[1] {\textit{\textbf{#1}}}

\newcommand{\R}{\mathbb{R}}

\newcommand{\Z}{\mathbb{Z}}

\begin{document}
\renewcommand{\ib}{\mathbf}
\renewcommand{\proofname}{Proof}
\renewcommand{\phi}{\varphi}
\renewcommand{\epsilon}{\varepsilon}
\newcommand{\conv}{\mathrm{conv}}

\title[]{On Helly numbers for crystals and cut-and-project sets}
\author{Alexey~Garber}

\address{School of Mathematical \& Statistical Sciences, The University of Texas Rio Grande Valley, 1 West University blvd., Brownsville, TX 78520, USA.}
\email{alexey.garber@utrgv.edu}


\date{\today}

\begin{abstract}
We prove existence of Helly numbers for crystals and for cut-and-project sets with convex windows. Also we show that for a two-dimensional crystal consisting of $k$ copies of a single lattice the Helly number does not exceed $k+6$.
\end{abstract}

\maketitle

\section{Introduction}

The Helly theorem \cite{Hel} states that if every $d+1$ or fewer sets from a finite family $\mathcal{F}$ of convex sets in $\R^d$ have non-empty intersection, then all sets in $\mathcal{F}$ have non-empty intersection. In other words, it claims that the Helly number of $\R^d$ is at most $d+1$. Similarly, the following theorem claims that the Helly number of the $d$-dimensional lattice $\Z^d=\{(a_1,\ldots,a_d)| a_i\in\Z\}$ is at most $2^d$. It is not complicated to show that $d+1$ and $2^d$ are also lower bounds for Helly numbers of $\R^d$ and $\Z^d$, respectively.

\begin{thm}[J.-P. Doignon, \cite{Doi}]\label{thm:doi} Let $\mathcal{F}$ be a finite family of convex sets in $\R^d$. If every $2^d$ sets from $\mathcal{F}$ intersect at a point of $\Z^d$, then all sets from $\mathcal{F}$ intersect at a point of $\Z^d$.
\end{thm}

This theorem is generalized in the following fractional version.

\begin{thm}[I. B\'ar\'any, J. Matou\v{s}ek, \cite{BM}]\label{thm:lattice}
For every $d\geq 1$ and every $\alpha\in(0,1]$ there exists a constant $\beta =\beta(d,\alpha)>0$ with the following property. Let $\mathcal{F}$ be a family of $n$ convex sets in $\R^d$ such that at least $\alpha\binom{n}{d+1}$ subfamilies of $d+1$ sets intersect at a point from $\Z^{d+1}$. Then there exists a point of $\Z^d$ contained in at least $\beta n$ sets from $\mathcal{F}$.
\end{thm}

The goal of this paper is to prove Helly-type theorems for crystals and for certain cut-and-project sets, the discrete point sets defined in Section \ref{crystals}. The fractional versions of the Helly-type theorem for these families will follow from a theorem by Averkov and Weismantel \cite{AW}.

This paper is organized as follows. In Section \ref{definitions} we introduce main definitions. In Section \ref{crystals} we prove existence of Helly numbers for every $d$-dimensional crystal and for every cut-and-project set with convex window. 
Finally, in Section \ref{2dim} we give exact maximum value of the Helly number for a two-dimensional $k$-crystal, i.e. $k$ translational copies of a fixed two-dimensional lattice.

\section{Helly numbers for arbitrary sets}\label{definitions}

First we introduce the $S$-Helly number of an arbitrary set $S\subset \R^d$. 

\begin{defin}
Let $S$ be a non-empty subset of $\R^d$. We define the {\it $S$-Helly number}, or $h(S)$, to be the minimal number $n$ that satisfies the the following statement. For every finite family $\mathcal{F}$ of convex sets in $\R^d$ with at least $n$ sets, if every $n$ sets from $\mathcal{F}$ intersect at a point of $S$, then all sets from $\mathcal{F}$ intersect at a point of $S$.

\end{defin}

In particular, the classical Helly theorem means that $h(\R^d)= d+1$ and Doignon's theorem means that $h(\Z^d)= 2^d$. More results on Helly numbers for various sets can be found in \cite{ADS,DLOR}.

It is not very hard to see that for an arbitrary set $S$ its Helly number does not necessary exist (or we can say that $h(S)=\infty$) even if $S$ is a discrete set. If for any $n\geq 3$, $S$ has $n$ points forming a convex $n$-gon without additional points of $S$ inside or on the boundary, then the following lemma shows that $h(S)\geq n$ for any prescribed $n$, so $h(S)$ cannot be finite.

\begin{lem}[G. Averkov, {\cite[Thm. 2.1]{Ave}}]\label{lem:empty}
Assume $S\subset \R^d$ is discrete. Then $h(S)$ is equal to the following two numbers:
\begin{enumerate}
\item The supremum of the number of facets of a convex polytope $P$ such that each facet of $P$ contains exactly one point from $S$ in its relative interior, and no other points of $S$ are contained in $P$.
\item The supremum of the number of vertices of a convex polytope $Q$ with vertices in $S$ that does not contain any other points from $S$.
\end{enumerate}
\end{lem}

One can construct such a set $S$ to be even a Delone set, i.e. a discrete set in $\R^d$ such that every sufficiently large ball in $\R^d$ contains at least one point from this set and every sufficiently small ball in $\R^d$ contains at most one point from this set. To do this in $\R^2$ equipped with the standard Cartesian coordinates, we take the lattice $\Z^2$ and for every $n\geq 3$ add a thin $n$-gon ``almost'' on the line $x=\frac12+n$. This example can be easily generalized to $\R^d$ for every $d\geq 2$.


\section{Crystals and cut-and-project sets}\label{crystals}

\begin{defin}
Let $\Lambda$ be a full rank lattice in $\R^d$, i.e. $\Lambda$ is the set of integer linear combinations of a basis of $\R^d$. Any union of finitely many translations of $\Lambda$ is called a {\it $d$-dimensional crystal}. To be more precise, if $\mathbf{t}_1,\ldots,\mathbf{t}_k$ are vectors in $\R^d$ such that for every $i\neq j$, the difference $\mathbf{t}_i-\mathbf{t}_j\notin \Lambda$, then the set 
$$\bigcup\limits_{i=1}^k\{\Lambda+\mathbf{t}_i\}$$
is called a {\it $d$-dimensional $k$-crystal}.
\end{defin}

Crystals are the only periodic Delone sets in $\R^d$. By periodic set in $\R^d$ we mean a set which symmetry group  possesses $d$ linearly independent translations. 

Existence of a Helly number for every crystal follows from the following lemma \cite[Prop.~2.8]{ADS}.

\begin{lem}\label{lem:union}
If $S_1,S_2\subseteq \R^2$, then $h(S_1\cup S_2)\leq h(S_1)+h(S_2)$.
\end{lem}

\begin{cor}\label{thm:crystals}
If $S$ is a $d$-dimensional $k$-crystal, then $h(S)\leq k2^d$.
\end{cor}
\begin{proof}
It is clear that the Helly number does not change if we apply an affine transformation to the set under consideration and thus for every $d$-dimensional lattice $\Lambda$, $h(\Lambda)=2^d$. Since $S$ can be expressed as $\bigcup\limits_{i=1}^k\{\Lambda+\mathbf{t}_i\}$, Lemma \ref{lem:union} implies that $$h(S)\leq \sum_{i=1}^kh(\Lambda+\mathbf{t}_i)=\sum_{i=1}^k2^d=k2^d$$
as claimed.
\end{proof}

In Section \ref{2dim} we show that this trivial bound is not sharp, see Theorem \ref{bound2dim} and Corollary~\ref{boundanydim}.

We will also be interested in another family of Delone sets, the cut-and-project sets. We give a slightly simplified definition here using only Euclidean spaces. We refer to \cite{BG,Moo} for a more general treatment.

\begin{defin}\label{def:cnp}
Let $\R^{d+k}$ be represented as $\R^{d+k}=\R^d\times\R^k$ and let $\Lambda$ be a lattice in $\R^{d+k}$. Let $\pi_1:\R^{d+k}\to \R^d$ and $\pi_2:\R^{d+k}\to \R^k$ be two projections on the complementary subspaces such that $\pi_1 |_{\Lambda}$ is injective, and $\pi_2(\Lambda)$ is dense in $\R^k$. Let $W \subset \R^k$ be a relatively compact set with non-empty interior --- the {\em window}. 

Then
$$V := V(\R^d,\R^k,\Lambda,W) = \{ \pi_1(\mathbf{x}) \, | \, \mathbf{x} \in \Lambda, \;
\pi_2(\mathbf{x}) \in W \}$$ is called a {\em cut-and-project set}, or a {\it model set}.

This is summarized in
the following diagram, which is called {\em cut-and-project scheme}.
$$
 \begin{array}{ccc}
\R^d & \stackrel{\pi_1}{\longleftarrow} \R^d \times \R^k
  \stackrel{\pi_2}{\longrightarrow} & \R^k \\
\cup & \cup & \cup \\
V & \Lambda & W
\end{array}
$$


\end{defin}


An example of cut-and-project is the one-dimensional Fibonacci tiling $F=F(\R,\R,\Lambda,W)$, see Figure \ref{pict:fib}. Here $\tau=\frac{\sqrt 5+1}{2}$ and $\Lambda$ is the lattice generated by the vectors $(1,1)$ and $(\tau, -\tau^{-1})$ and window $W$ is the half-open interval $[-\tau^{-1},1)$ while $\pi_1$ and $\pi_2$ are orthogonal projections. 

\begin{center}
\begin{figure}[!ht]
\includegraphics[width=\textwidth]{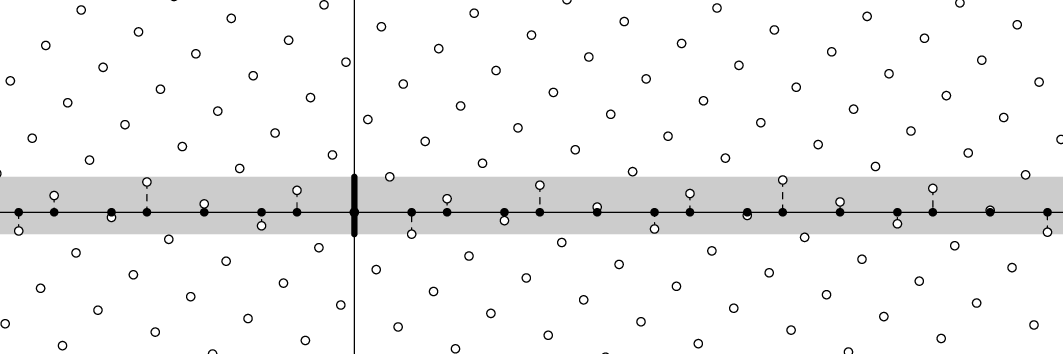}
\caption{Cut-and-project construction of the Fibonacci tiling. The bold vertical segment represents the window, the shaded strip ``cuts'' the points with $\pi_2$-projections in the window, and the solid points on the horizontal axis are the $\pi_1$-projections of the points in the strip.}
\label{pict:fib}
\end{figure}
\end{center}

More examples of cut-and-project sets include (vertex sets of) Penrose tilings \cite{deB}, Ammann-Beenker tilings \cite{Bee}, and many others, see \cite{FH} for more examples. Certain local patches of Penrose and Ammann-Beenker tilings are shown in Figure \ref{pict:PandAB}. In fact, the Penrose tiling as described in \cite{deB} is not a cut-and-project set in the sense of Definition \ref{def:cnp}. The cut-and-project scheme described in \cite{deB} uses the five-dimensional lattice $\Z^5$ with a non-dense $\pi_2$-image in the window given by a three-dimensional polytope. More precisely, $\pi_2(\Z^5)$ is dense in a union of four polygons inside the window. However, such a ``degeneracy'' if one occurs, will not affect our proof, so we will not emphasize separately whether $\pi_2$ creates a dense image of $\Lambda$ or not.

Cut-and-projects sets play a central role in the study of aperiodic order and serve as a well-established model for aperiodic crystals. We refer to \cite{Moo} and \cite{BG} for surveys on model sets, their properties, and more references.

\begin{center}
\begin{figure}[!ht]
\includegraphics[width=0.4\textwidth]{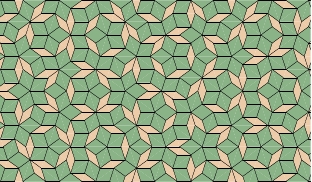}
\hskip 0.1\textwidth
\includegraphics[width=0.4\textwidth]{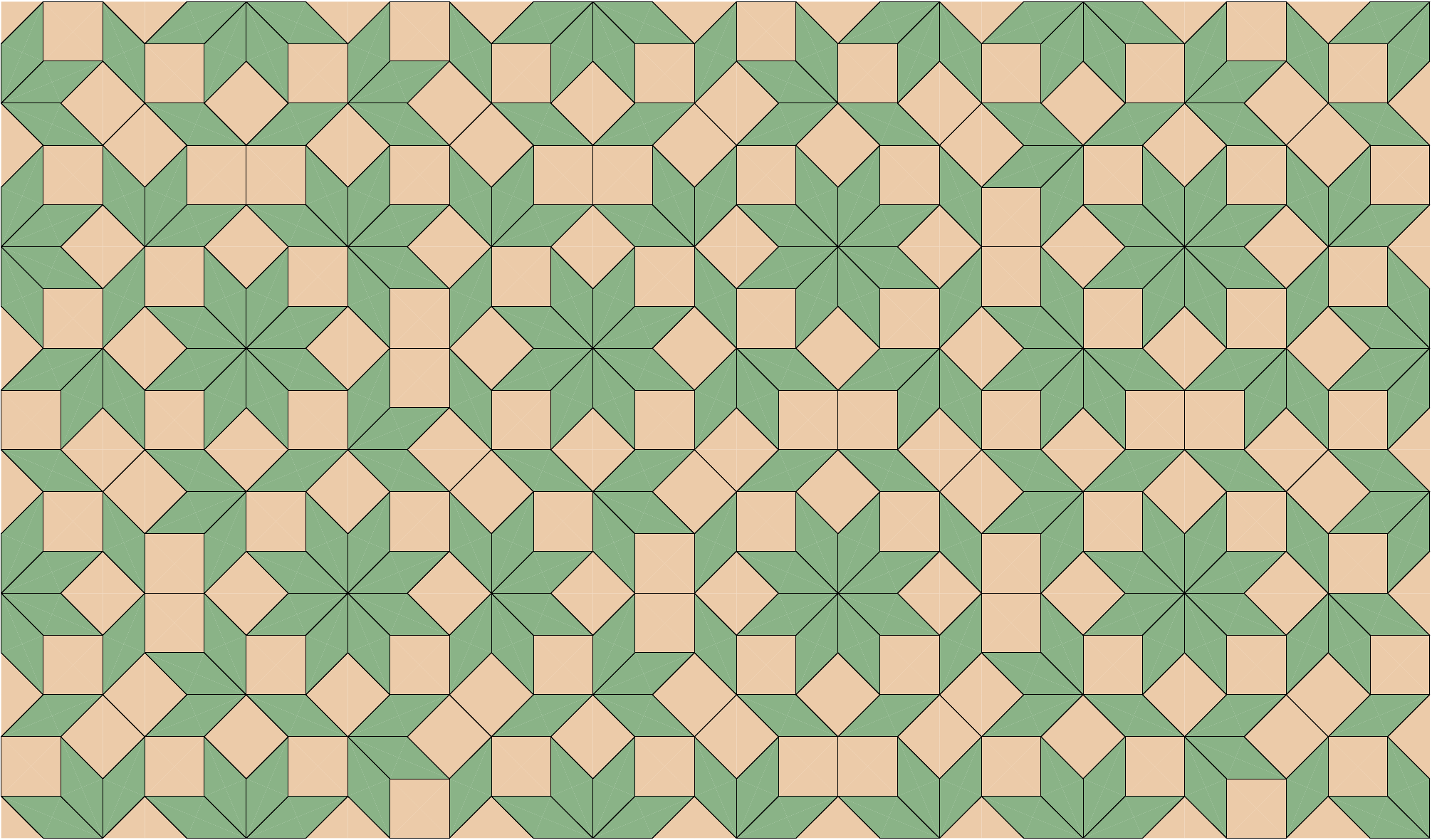}
\caption{Penrose and Ammann-Beenker tilings. The images are courtesy of Dirk Frettl\"oh.}
\label{pict:PandAB}
\end{figure}
\end{center}

Our next goal is to prove that every cut-and-project set with a convex window $W$ has a finite Helly number.

\begin{thm}\label{thm:cnp}
Let $V=V(\R^d,\R^k,\Lambda,W)$ be a $d$-dimensional cut-and-project set with a convex $k$-dimensional window $W$. Then $h(V)\leq 2^{d+k}$.
\end{thm}

\begin{proof}
Let $\mathcal{F}$ be a finite family of convex $d$-dimensional sets such that every $2^{d+k}$ sets in $\mathcal{F}$ intersect at a point of $V$. We will show that all sets from $\mathcal{F}$ intersect at a point of $V$ which will be enough to prove the theorem.

Let $\pi_1$ and $\pi_2$ be the projections used in the construction of $V$. By $\pi_1^{-1}(X_1)$ and $\pi_2^{-1}(X_2)$ we will denote  preimages of sets $X_1$ and $X_2$ with respect to $\pi_1$ and $\pi_2$, respectively. 
For every set in $\mathcal F$, we construct its $\pi_1$-preimage and intersect with the $\pi_2$-preimage of the window. In other words, we let $\mathcal{F}'=\{\pi_1^{-1}(F)\cap\pi_2^{-1}(W)|F\in\mathcal{F}\}$. Since $W$ is convex and the preimage of a convex set with respect to every projection is convex, $\mathcal{F}'$ is a finite family of convex sets in $\R^{d+k}$.

Let $F'_1, \ldots, F'_{2^{d+k}}$ be any sets from $\mathcal{F'}$. Their projections $\pi_1(F'_1), \ldots, \pi_1(F'_{2^{d+k}})$ are sets from $\mathcal{F}$ so they intersect at a point $\mathbf{x}\in V$. By the construction of $V$, there is a (unique) point $\mathbf{x}'\in \Lambda$ such that $\pi_1(\mathbf{x}')=\mathbf{x}$ and $\pi_2(\mathbf{x}')\in W$. Therefore $\mathbf{x}'$ belongs to every $F'_i$ for $i=1,\ldots,2^{d+k}$, and every $2^{d+k}$ sets from $\mathcal{F}'$ intersect at a point of $\Lambda$. Since $h(\Lambda)=2^{d+k}$, all sets from $\mathcal{F}'$ intersect at a point $\mathbf{x}'_0\in \Lambda$. 
The projection $\mathbf{x}_0=\pi_1(\mathbf{x}'_0)$ is a point of $V$ because $\pi_2(\mathbf{x}'_0)\in W$. This projection $\mathbf{x}_0$ belongs to every set from $\mathcal{F}$, therefore all sets from $\mathcal{F}$ intersect at a point of $V$.
\end{proof}

Using Lemma \ref{lem:empty} we get the following corollary for every cut-and-project set with convex window.

\begin{cor}\label{cor:empty}
If $V$ is a $d$-dimensional cut-and-project set with a convex $k$-dimensional window, then every convex polytope with at least $2^{d+k}+1$ vertices in $V$ has an additional point of $V$ inside or on the boundary.
\end{cor}

For the special case of the Penrose tilings, we get the following estimate.

\begin{cor}
If $P$ is the set of vertices of a Penrose tiling, then $h(P)\leq 32$. In particular, every convex polytope with $33$ vertices in $P$ has an additional vertex of the same Penrose tiling inside or on the boundary.
\end{cor}
\begin{proof}
Sections 7 and 8 of \cite{deB} establish that vertices of a Penrose tiling form a two-dimensional cut-and-project set with a three-dimensional window given by a convex polytope. Applying Theorem \ref{thm:cnp} we get $h(P)\leq 2^{2+3}=32$. The second statement of the corollary is just a reformulation of Corollary \ref{cor:empty} for the Penrose tiling.
\end{proof}




Theorem \cite[Thm. 1.3]{AW} states that if the Helly number of a closed subset $S\subseteq \R^d$ is finite, then there is a fractional version of the Helly theorem for $S$. Additionally, the fractional Helly number is at most $d+1$ meaning that existence of any fraction of subfamilies of $d+1$ sets in $S$ with nonempty intersections implies existence of another fraction of sets in $S$ with a nonempty intersection. Corollary \ref{thm:crystals} and Theorem \ref{thm:cnp} guarantee that this can be applied to crystals and cut-and-project sets as both classes consist of closed sets.

\begin{cor}
Let $S\subset \R^d$ be a crystal or a cut-and-project set with a convex window. Then the fractional Helly number of $S$ is at most $d+1$. 

That is, for every $\alpha\in(0,1]$ there exists a constant $\beta =\beta(S,\alpha)>0$ with the following property. Let $\mathcal{F}$ be a family of $n$ convex sets in $\R^d$ such that at least $\alpha\binom{n}{d+1}$ subfamilies of $d+1$ sets intersect at a point from $S$. Then there exists a point of $S$ contained in at least $\beta n$ sets from $\mathcal{F}$.
\end{cor}

\section{Helly numbers for two-dimensional crystals}\label{2dim}

In the final section we give a sharp bound for the Helly numbers of two-dimensional crystals depending on the number of copies of the lattice used. Corollary \ref{thm:crystals} says that the Helly number of a two-dimensional $k$-crystal is not greater than $4k$. Below we will show that it is not greater than $k+6$ for $k\geq 6$, and even smaller for other $k$. Without loss of generality for all crystals below we will use $\Z^2$ as the generating lattice.

We start from an example of a $k$-crystal with Helly number at least $k+6$ and later show that it is optimal for $k\geq 6$. Throughout the section we refer to the second part of Lemma \ref{lem:empty}, so for a given crystal $S$ we will be interested only in empty convex polygons with vertices from $S$ with maximal number of vertices. 

\begin{exam}\label{exam:k+6}
Figure \ref{pict:k+6} shows an example of an empty $12$-gon within a $6$-crystal obtained as six translations of the interger lattice $\Z^2$.  The integer points are marked with blue in this picture.

If we want to construct a crystal with more than 6 copies of $\Z^2$, then we can add any number of additional copies translated by vectors represented by points on the red circular arc next to the 12-gon in the same Figure \ref{pict:k+6}. Note, that if we add $k-6$ additional points (so the resulting crystal will be a $k$-crystal), then we will be able to find an empty convex $(k+6)$-gon (12 initial points and $k-6$ points on the red arc immediately to the right of it), so the resulting crystal will have Helly number at least $k+6$.



\begin{center}
\begin{figure}[!ht]
\includegraphics[width=0.5\textwidth]{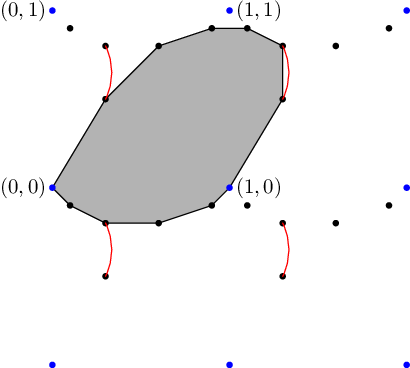}
\caption{Two-dimensional $6$-crystal with Helly number $12$. The red arcs allow construction of $k$-crystals with Helly number $k+6$ for every $k\geq 6$.}
\label{pict:k+6}
\end{figure}
\end{center}

\end{exam}

In the remaining part of the paper we will show that the previous example is optimal in case there are at least 6 copies of $\Z^2$. We start from two simple auxiliary lemmas. 

\begin{lem}\label{lem:4gon}
Every convex quadrilateral $Q$ is either a parallelogram, or contains a parallelogram $Q'$ such that three vertices of $Q'$ are vertices of $Q$.
\end{lem}

\begin{proof}
Without loss of generality we can assume that three vertices of the quadrilateral are points $A=(0,1)$, $B=(0,0)$, and $C=(1,0)$, and the fourth vertex $D=(x,y)$ satisfies $x,y>0$ and $x+y>1$.

If $x>1$ and $y>1$, then the parallelogram $ABCM$ with $M=(1,1)$ will be inside $ABCD$. If $x>1$ and $y\leq 1$, then the parallelogram $BCDM$ with $M=(x-1,y)$ will be inside $ABCD$. If $x\leq 1$ and $y>1$, then the parallelogram $ABMD$ with $M=(x,y-1)$ will be inside $ABCD$. If $x\leq 1$ and $y\leq 1$ then $ABCD$ is either a parallelogram, or contains the parallelogram $AMCD$ with $M=(1-x,1-y)$.
\end{proof}

\begin{lem}\label{lem:parallel}
Let $XX'$, $YY'$ and $ZZ'$ be three equal and parallel segments. Then $\conv(X,Y,Z,X',Y',Z')$ contains at least one of these six points inside or on the boundary.
{\sloppy 

}
\end{lem}

\begin{proof}
If two of the segments $XX'$, $YY'$ and $ZZ'$ lie on one line, then the lemma is trivial. Otherwise, without loss of generality we can assume that points have the following coordinates: $X=(0,0)$, $X'=(1,0)$, $Y=(0,1)$, $Y'=(1,1)$, $Z=(x,y)$, $Z'=(x+1,y)$ for $y>1$.

If $x>0$, then $Y'$ lies in $\conv(X,X',Y,Z,Z')$. Otherwise, $Y$ lies in $\conv(X,X',Y',Z,Z')$.
\end{proof}

\begin{thm}\label{bound2dim}
Let $S$ be a $2$-dimensional $k$-crystal. Then $h(S)\leq k+6$.
\end{thm}
\begin{proof}
Let $A_1,A_2\ldots,A_n$ be points in $S$ with maximal possible $n$ such that $P=\conv(A_1,A_2\ldots,A_n)$ is a convex polygon without additional points of $S$ inside or on the boundary. According to Lemma \ref{lem:empty}, it is enough to show that $n\leq k+6$.
{\sloppy

}

Let $N$ be the maximal number of vertices of $P$ that belong to one copy of $\Z^2$. We will study all possible cases of $N$ and find a bound for $n$ for each case.

\subsection*{Case 1: $N\geq 5$.} This case is impossible as according to Lemma \ref{lem:empty} combined with Theorem \ref{thm:doi} for $d=2$, the convex hull of every $5$-gon with vertices at integer points will have an integer point inside or on the boundary since $h(\Z^2)=4$.

\subsection*{Case 2: $N=4$.} 



Let $A,B,C,D$ be four vertices of $P$ from one copy of $\Z^2$. If $\conv(A,B,C,D)$ is not a parallelogram, then the fourth vertex of the parallelogram from Lemma \ref{lem:4gon} is also from the same copy of $\Z^2$ and $P$ is not empty.

If $\conv(A,B,C,D)$ is a parallelogram, then it is a fundamental parallelogram of $\Z^2$ because it is empty, and thus it will have a point from every translation of $\Z^2$ inside or on the boundary. Therefore, if $k\geq 2$, then this case is impossible. We will collect maximal values of $n$ for various $k$ in the following table:
$$
\begin{array}{l|r|c|c|c|c|c|c|}
\text{Case }2:N=4&k&1&2&3&4&5&\geq 6\\\hline
&\text{max. }n&4&-&-&-&-&-
\end{array}
$$

\subsection*{Case 3: $N=3$.} Let $A,B,C$ be three vertices of $P$ from one copy of $\Z^2$. The area of the triangle $ABC$ must be $\frac12$ or it will be non-empty otherwise. Applying an affine transformation we can make $A=(0,0)$, $B=(1,0)$ and $C=(1,1)$ as every lattice triangle with area $\frac12$ can be transformed into this one.

Assume $P$ has a vertex $M=(x,y)$ with $y>1$. Then the triangle $ABM$ is not empty because it contains either the integer point $(\lfloor \frac{x}{y}\rfloor+1,1)$ or the point $(x-\lfloor \frac{x}{y}\rfloor,y-1)\in M+\Z^2$. Similarly $P$ can't have a vertex $M=(x,y)$ with $x<0$ (then the triangle $BCM$ will be non-empty) or with $x-y>1$ (then the triangle $CAM$ will be non-empty). Therefore, all vertices of $P$ are inside the triangle with vertices $(0,-1)$, $(2,1)$, $(0,1)$.

Assume there is one more copy of $\Z^2$ which contains three vertices of $P$. Then these vertices should be $D=(x,y)$, $E=(x+1,y)$, and $F=(x,y-1)$, see Figure \ref{pict:3+3}.
\begin{center}
\begin{figure}[!ht]
\includegraphics[width=0.5\textwidth]{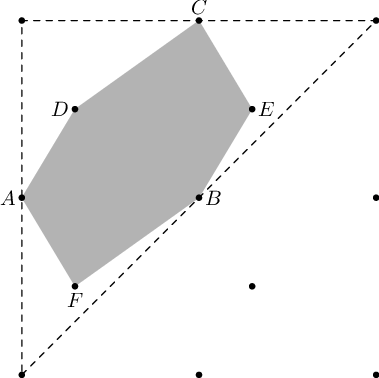}%
\caption{$P$ has three vertices from two copies of $\Z^2$.}%
\label{pict:3+3}%
\end{figure}
\end{center}
Now we can see that translations of the hexagon $AFBECD$ cover the plane, and there could not be other copies of $\Z^2$ in the crystal $S$. Also $P$ can't have more than 6 vertices as all the points of $S$ inside the triangle with vertices $(0,-1)$, $(2,1)$, $(0,1)$ are already vertices of $P$.

For the rest of this case all other copies of $\Z^2$ have at most 2 points among vertices of $P$. Note, that if a copy of $\Z^2$ has two points among vertices of $P$, then the vector connecting these vertices should be $(0,1)$, $(1,0)$, or $(1,1)$, otherwise both points can't be inside the triangle with vertices $(0,-1)$, $(2,1)$, $(0,1)$. This, together with Lemma \ref{lem:parallel} will let us bound the number of copies of $\Z^2$ with two points among vertices of $P$. 

%

If there are at least four additional copies of $\Z^2$ with two points among vertices of $P$, then at least two corresponding pairs of segments are equal and parallel, and applying Lemma~\ref{lem:parallel} to these four points and a parallel side of the triangle $ABC$ we get a contradiction. Thus, if $k\geq 4$, then all copies of $\Z^2$, except possibly four, can have at most one point among vertices of $P$. If $k$ is 3, then one copy can have three points, and two other can have two points each.

Altogether, the third case can be summarized in the following table:
$$
\begin{array}{l|r|c|c|c|c|c|c|}
\text{Case }3:N=3&k&1&2&3&4&5&\geq 6\\\hline
&\text{max. }n&3&6&7&9&10&k+5
\end{array}
$$

\subsection*{Case 4: $N=2$.} We will show that there are at most $6$ copies of $\Z^2$ that have exactly two points among vertices of $P$. We will prove that statement by contradiction and we assume that there are at least $7$ copies of $\Z^2$ that have two points among vertices of $P$ each. 

Every vector connecting two vertices of $P$ from the same copy of $\Z^2$ must be a basis vector of $\Z^2$. According to Lemma \ref{lem:parallel}, there could be at most two vectors of any direction, so there will be vectors of at least four directions. We claim that among these four vectors there exist two that do not form a basis of $\Z^2$. 

Again, looking for contradiction, If every pair forms a basis, we can assume that one pair is the standard basis $(1,0)$ and $(0,1)$. Then each coordinate of the remaining vectors should be $1$ or $-1$, but any two vectors of the form $(\pm1,\pm1)$ do not form a basis of $\Z^2$ since the corresponding determinant is 0 or $\pm 2$. Thus, there are two vectors among the initial four that do not form a basis of $\Z^2$.

Without loss of generality we may assume that points $(0,0)$ and $(1,0)$ form one vector of the two, and the second vector is formed by points $(x,y)$ and $(x',y')$ with $|y-y'|\geq 2$. Also, without loss of generality we may assume that $y\geq 1$. Similarly to an argument in the third case, the polygon $P$ will not be empty as it will contain the point $(\lfloor \frac{x}{y}\rfloor+1,1)$ or the point $(x-\lfloor \frac{x}{y}\rfloor,y-1)$. This gives the final contradiction and we have established that at most $6$ copies of $\Z^2$ are represented by two points among vertices of $\Z^2$; all other $k-6$ copies (if $k\geq 6$) are represented by at most one vertex of $P$.



Thus, the table for Case 3 is the following (case $k=1$ is impossible here, as $P$ can't be a $2$-gon):
$$
\begin{array}{l|r|c|c|c|c|c|c|}
\text{Case }4:N=2&k&1&2&3&4&5&\geq 6\\\hline
&\text{max. }n&-&4&6&8&10&k+6
\end{array}
$$

\subsection*{Case 5: $N=1$.}
In that case we can write the table immediately:
$$
\begin{array}{l|r|c|c|c|c|c|c|}
\text{Case }5:N=1&k&1&2&3&4&5&\geq 6\\\hline
&\text{max. }n&-&-&3&4&5&k
\end{array}
$$

Summarizing all tables we can see that $h(S)\leq k+6$ for every $k$.
\end{proof}

In addition to the previous theorem we can find a sharp bound for $h(S)$ for every $k$.

If $k\geq 6$ then $h(S)\leq k+6$ and the example with $h(S)= k+6$ for any $k\geq 6$ is described in Example \ref{exam:k+6}. If $k=5$, then $h(S)\leq 10$ and for the equality we can take Example \ref{exam:k+6} and remove one copy of $\Z^2$. If $k=4$ then $h(S)\leq 9$ and an example with equality consist of $\Z^2$ shifted by $(0,0)$, $(\frac{2}{10},\frac{3}{10})$, $(\frac{7}{10},\frac{8}{10})$, $(\frac{12}{10},\frac{8}{10})$, see Figure \ref{pict:k=4}. If $k=3$, then $h(S)\leq 7$ and for an example with equality we can take the crystal for $k=4$ and remove the copy of $\Z^2$ translated by $(\frac{12}{10},\frac{8}{10})$. If $k=2$, then $h(S)\leq 6$ with example shown in Figure \ref{pict:3+3}. Finally, if $k=1$ then $S$ is a lattice and $h(S)$=4.

\begin{center}
\begin{figure}[!ht]%
\includegraphics[width=0.5\textwidth]{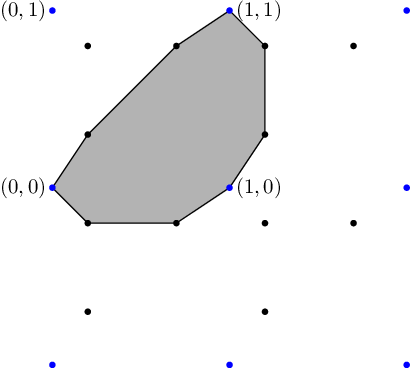}%
\caption{$4$-crystal with $h(S)=9$.}%
\label{pict:k=4}%
\end{figure}
\end{center}

Altogether this can be summarized in the following table:
$$
\begin{array}{r|c|c|c|c|c|c|}
k&1&2&3&4&5&\geq 6\\\hline
\text{max. }n&4&6&7&9&10&k+6
\end{array}
$$

As the last result we discuss some bounds in higher dimensions.

\begin{cor}\label{boundanydim}
For every $k\geq 6$ there exists a $d$-dimensional $k$-crystal $S$ such that $h(S)\geq 2^{d-2}(k+6)$
\end{cor}
\begin{proof}
We can take the crystal from Example \ref{exam:k+6} and construct its direct product with $\Z^{d-2}$. Then the product of the empty polygon with $k+6$ vertices with the cube $[0,1]^{d-2}$ will give an empty polytope with $2^{d-2}(k+6)$ vertices.
\end{proof}

Thus, if $H(d,k)$ denotes the maximal Helly number among $d$-dimensional $k$-crystals (with large $k$), then $2^{d-2}(k+6)\leq H(d,k)\leq k2^d$ and $H(d,k)$ is linear in $k$ and exponential in $d$.

\section*{Acknowledgments}

The author is thankful to Pablo Sober\'on for discussion that led to this paper and for valuable comments. The author is also thankful to Alexey Glazyrin for valuable discussions. It is worth mentioning that he found a less technical proof of Theorem \ref{bound2dim} in private communication \cite{Gla}. Finally, the author is thankful to Dirk Frettl\"oh for help with pictures of aperiodic tilings.



\end{document}